\documentclass{article}
\usepackage{amsmath,amsthm,amssymb,amsfonts}
\usepackage{enumerate}
\usepackage{hyperref}
\usepackage[capitalise]{cleveref}
\usepackage[loadshadowlibrary]{todonotes}
\setuptodonotes{inline, color=red!100!green!33, shadow}
\usepackage{graphicx}
\usepackage{comment}
\usepackage{stmaryrd}
\usepackage[T1]{fontenc}
\usepackage[utf8]{inputenc}
\usepackage{xurl}

\newtheorem{theorem}{Theorem}

\newtheorem{definition}{Definition}
\newtheorem{example}{Example}
\newtheorem{lemma}{Lemma}

\newtheorem{notation}{Notation}
\newtheorem{remark}{Remark}

\title{Triangular tensors and set-intersection problems}

\author{Omran Ahmadi$^\dagger$ \and Hassan Norouzi$^*$}

\date{
\vspace*{3mm}
\parbox{\linewidth}{
\centering
\small
 $^\dagger$ School of Mathematics,\\ 
  Institute for Research in Fundamental Sciences (IPM),\\ 
  E-mail: \texttt{oahmadid@ipm.ir}
 \endgraf\medskip
  $^*$ School of Mathematics,\\ 
  Institute for Research in Fundamental Sciences (IPM),\\
  E-mail: \texttt{norouzi@ipm.ir}
  \endgraf\medskip
  \today
  }
}

\begin{document}
\maketitle

\begin{abstract}
In the past few years, the slice-rank lemma of Tao has been applied successfully to many problems in extremal combinatorics. In this paper, first, we define a new notion of triangular tensors which generalizes that of triangular matrices (2-tensors), and prove a lemma similar to the slice-rank lemma for them. Then, applying the slice-rank framework with triangular matrices, we give new and shorter proofs for some well-known theorems on set-intersections like Frankl-Wilson and Snevily with modular constraints, and some of the more recent set-intersection results. We also improve Snevily with modular constraints in some special cases. Finally, using Snevily's theorem with some combinatorial lemmas, we give new bounds on some generalizations of the reverse odd-town problem. 
\end{abstract}

\section{Introduction}
In the extremal set theory, one is usually interested in finding the minimum or maximum number of subsets of a finite set subject to some constraints. Pioneering results in the extremal set theory were obtained in the first half of the twentieth century.  To name some of them,  one can include the results of Mantel~\cite{Mantel}, Sperner~\cite{Sperner}, Tur\'an~\cite{Turan} and Erd\H{o}s-Ko-Rado~\cite{Erdos-Ko-Rado} (EKR for short), though the results of Mantel and Tur\'an are usually stated in the language of graphs rather than sets. Among these four results, the EKR theorem, which states that the maximum number of $r$-element subsets of an $n$-element set with pairwise non-empty intersections for $r \leq n/2$ is $\binom{n-1}{r-1}$, belongs to a family of extremal results called set-intersection results. 

In set-intersection problems, one is usually interested in finding the maximum number of subsets of a finite set subject to some restrictions on their intersections. The EKR theorem, since its publication in 1961 (originally proved in 1938), has been inspirational to the extent that the results and proof techniques on the set-intersection problems have grown to a large body of work~\cite{Frankl-book,Frankl-survey-JCTA-2016}. 

In the most general version of the set-intersection problems, one is given an $n$-element set $S$, for example $[n]:=\{1,2,\ldots,n\}$, a collection $L_1,\ldots,L_k$  of  $k$ non-empty subsets of $[n]$, and looks for the  greatest $m$ for which there is a family  \( \mathcal{F} = \{A_1, \ldots, A_m\} \) of $m$ subsets of $S$ such that 
\[
    |A_{r_1} \cap \cdots \cap A_{r_i}|  \in L_i 
\]
for all  \( 1 \le r_1 < \cdots < r_i \le m \) with \( 1 \le i \leq k\) where $|\cdot|$ denotes the number of elements of a set. Usually, when one is concerned with the remainders of the sizes of the sets in $\mathcal{F}$ and their intersections modulo a fixed prime number $p$ and not their actual sizes,  $|L_1|\pmod{p},\ldots,|L_k|\pmod{p}$ are important and not the absolute values $|L_1|,\ldots,|L_k|$. Thus, in these cases, called modular variants, without loss of generality,  $L_1,\ldots,L_k$ may be taken to be $k$ non-empty subsets of $\mathbb{Z}_p$.

Depending on $L_1,\ldots,L_k$, the corresponding extremal problem may be trivial or may be very challenging and interesting. For example, in EKR, $k=2$, $L_1=\{r\}$ and $L_2=\{1,2,\ldots,r-1\}$. Some of the well-known results on the pairwise intersections and hence $k=2$ are:  odd-town theorem conjectured by Erd\H{o}s and settled by Berlekamp~\cite{Berlekamp} where $L_1=\{1\}\subset \mathbb{Z}_2$ and $L_2=\{0\}\subset \mathbb{Z}_2$, reverse odd-town theorem where $L_1=\{0\}\subset \mathbb{Z}_2$ and $L_2=\{1\}\subset \mathbb{Z}_2$ (see Exercise~1.1.5 in~\cite{BabaiFrankl1992}), the Frankl-Wilson theorem~\cite{FranklWilson1981} where $L_1=[n]$ and $L_2$ is a set of $s$ arbitrary non-negative integers not greater than $n$, and the theorem of Snevily with modular constraints~\cite{Snevily1} where again $S=[n]$ and  $L_1$ and $L_2$  are two disjoint subsets of $\mathbb{Z}_p$ for some prime number $p$. 

The proof techniques which have been developed to attack set-intersection problems now include algebraic, combinatorial and probabilistic among others.  One of the relatively recent added linear-algebraic tools to tackle the extremal combinatorics problems is the so-called {\it{slice-rank method}}. The slice-rank method has been very successful in obtaining better bounds for some notoriously hard problems. To give a brief account of the method, suppose $A$ is a finite set whose size we wish to bound and $F$ is a field. In the slice-rank method, we first write a diagonal $k$-tensor $T$ where $T$ is a function from $A^k$ to $F$ with the property that $T(a_1, \ldots, a_k) \neq 0$ only when $a_1 = a_2 = \cdots = a_k$. Then, as the size of $A$ is equal to the slice-rank of $T$ by the so-called slice-rank lemma, we bound the slice-rank of $T$ which in some precise sense generalizes the notion of the rank of the matrices (2-tensors). 

Having the slice-rank method and lemma for diagonal tensors, one may wonder whether it is possible to extend the method and lemma to the other types of tensors. To answer this question, the most natural place to look at, is the case of 2-tensors, i.e., matrices, and actually there is a fact about matrices that inspired us to extend the lemma to the other types of tensors in this paper. The fact that has been our starting point in our research in this paper is the fact that triangular matrices with non-zero diagonal entries are full-rank. This fact led us to define the new notion of triangular tensors which generalizes that of triangular matrices (2-tensors), and prove a lemma similar to the slice-rank lemma for them when the ground set $A$ is totally ordered. 

After having the new notion of triangular tensors and slice-rank lemma for them, naturally, one may look for extremal combinatorics problems that can be attacked using this new slice-rank lemma. Though, we have been unable to solve any extremal combinatorics problem using triangular tensors of order greater than two so far, to our surprise there are some well-known extremal combinatorics problems that can be attacked using triangular 2-tensors, and it seems that applying this method to them has been overlooked. For example, using this method, we give new and shorter proofs for some well-known theorems on set-intersections like Frankl-Wilson and Snevily with modular constraints, and some of the more recent set-intersection results of Liu~\emph{et al.}~\cite{liu2024snevily} and Heged{\"u}s~\cite{hegedus2024upper}. 

Among the four theorems that we give new proofs, the theorem of Snevily with modular constraints is a generalization of both odd-town and reverse odd-town theorems though in the latter case it gives a slightly weaker bound. In this paper, using some combinatorial arguments, we improve Snevily with modular constraints so that it gives the exact bound of reverse odd-town theorem. Then, we use this new theorem with some combinatorial lemmas to give new bounds on some generalizations of the reverse odd-town problem. In the generalization of reverse odd-town problem that we consider, with the notations as above, we fix a prime number $p$ and take $L_1=\cdots=L_{k-1}=\{0\}$ and let $L_k$ be a symmetric subset of $\mathbb{Z}_p$. Our results in this case, ~\cref{Generalized-reverse-odd-town,sharper}, improve the bounds given in~\cite{ONeill2025}. 

The remainder of this paper is organized as follows. We start~\cref{Pre} by listing required definitions and notation, then we give more details on the slice-rank method, and finally conclude the section by presenting our preliminary results. In~\cref{Main}, we present our main results including new proofs of the results mentioned above, a new result improving the modular variant of Snevily's theorem in some special cases, and finally our new bound for some generalizations of the reverse odd-town problem.

\section{Preliminaries}\label{Pre}
\subsection{Notation and definitions}
In this section, we gather some notation and definitions that will be used throughout the paper. First we notice that throughout the paper, small letters are used to denote the elements of a set, capital letters are used to denote a set, and calligraphic letters are reserved to denote a family of sets. Also, throughout the paper, boldface letters represent vectors.  

\begin{notation}
    We denote the set $\{1,2,\ldots,n\}$ by $[n]$, and the set of all subsets of $[n]$ by $2^{[n]}$.
\end{notation}

\begin{definition}
A vector $\boldsymbol{v}$ in a vector space with entries  equal to $0$ or $1$ is called a $\{0,1\}$-vector. 
\end{definition}

\begin{definition}
   The characteristic vector $\boldsymbol{v}_A\in \{0,1\}^n$ of a subset $A$ of $[n]$ is the $\{0,1\}$-vector whose $i$-th entry is $1$ if $i\in A$ and $0$ otherwise. Furthermore, the $i$-th entry of the characteristic vector of a set $A$ is denoted by $\boldsymbol{v}_{A}^{i}$.
\end{definition}

\begin{definition}
    The Hamming weight or simply the weight $w(\boldsymbol{v})$ of a vector $\boldsymbol{v}\in\{0,1\}^n$ is the number of its non-zero entries.
\end{definition}
%

\begin{notation}
    For a given set $L \subset \mathbb{Z}_p$, $-L:=\{-x \mid x \in L\}$.
\end{notation}

\begin{definition} Let $L_1,L_2,\ldots,L_k$ be a collection of subsets of  $[n]$ or a collection of subsets of $\mathbb{Z}_p$. Then a $k$-intersection profile is a $k$-tuple $\alpha=(L_1,L_2,\ldots,L_k)$. Whenever $L_i=\{\alpha_i\}$ for some $i$,  we write $\alpha=(L_1,\ldots,\alpha_i,\ldots,L_k)$ instead.
\end{definition}

\begin{definition}[{$[n,p,\alpha]$}-family]
Let  \( \alpha = (L_1, \ldots, L_k) \) be a $k$-intersection profile where $L_i\subseteq\mathbb{Z}_p$ for every $i$.  
A family  \( \mathcal{F} = \{A_1, \ldots, A_m\} \) of subsets of $[n]$ is called an \( [n, p, \alpha] \)-family if for all \( 1 \le r_1 < \cdots < r_i \le m \) with \( 1 \le i \leq k \), we have
\[
    |A_{r_1} \cap \cdots \cap A_{r_i}|\pmod{p} \in L_i . 
\]

\end{definition}

Note that by the above definitions it is implied that if $\alpha$ is a $k$-intersection profile and $\mathcal{F}$ is an $[n,p,\alpha]$-family, then $|\mathcal{F}|\ge k$.

\subsection{Slice-rank method}
The slice-rank method is a polynomial method rooted in a breakthrough paper by Lev, Croot and Pach~\cite{croot_progression-free_2016} where it was shown that for an integer \( n \geq 1 \) and any subset \( A \subseteq \mathbb{Z}_4^n \) free of non-trivial three-term arithmetic progressions one has \( |A| \leq 4^{\gamma n} \) for an absolute constant\( \gamma \approx 0.926 \). Shortly after the appearance of~\cite{croot_progression-free_2016} on ArXiv, Ellenberg and Gijswijt~\cite{ellenberg2016large} used the method of Lev, Croot and Pach to show that if $A$ is a subset of the $n$-dimensional vector space $\mathbb{F}_3^n$ over the ternary field $\mathbb{F}_3$ containing no non-trivial three-term arithmetic progression, then \(|A| = o(2.756^n). \)
Finding the largest subset of $\mathbb{F}_3^n$ with no three terms in arithmetic progression is the cap-set problem. Inspired by these developments, Tao~\cite{Tao-symmetric-reformulation} and later Tao and Sawin~\cite{tao_notes_2016} reformulated the method of Croot, Lev and Pach making it available for application to a wider range of problems. The reformulation of Tao and Sawin is currently known as the slice-rank method which is the subject of our short exposition in the sequel. 

Let $A$ be a finite set and $F$ be a field. A $k$-tensor $T$ is a function from $A^k$ to $F$.
A $k$-tensor $T$ is called diagonal if $T(a_1, \ldots, a_k) \neq 0$ only when $a_1 = a_2 = \cdots = a_k$, and it is called {\it{slice}} if it can be decomposed as \(T = f(a_j) g(a_1,\ldots,\hat{a_j},\ldots,a_k)\) for some $j$ in $[k]$, where as usual the notation \(g(a_1,\ldots,\hat{a_j},\ldots,a_n)\) or simply $g(\hat{a_j})$ means that the function $g$ does not depend on $a_j$. The important notion of the slice rank of a $k$-tensor is as follows. 
\begin{definition}[Slice Rank]
    Let $T\colon A^k \longrightarrow F$ be a $k$-tensor. The slice rank of  $T$ denoted by $sr(T)$ is the minimum $r$ such that $T$ can be written as the sum of $r$ slices, i.e., $sr(T)$ is the minimum number $r$ such that
    \[
        T= \sum_{i=1}^r f_i(a_{j_i}) g_i(\hat{a_{j_i}})
    \]
    for some $f_i\colon A \rightarrow F$ and $g_i\colon A^{k-1} \rightarrow F$.
    
\end{definition}
\begin{example}\label{innerproduct-slice rank}
Let $V$ be a subset of $\{0,1\}$-vectors of an $m$-dimensional vector space over the field $F$, and let $l\le m$. Then the slice rank of the $2$-tensor \[T(\boldsymbol{v}_1,\boldsymbol{v}_2)=\prod_{i=1}^l \left((\sum_{j=1}^m x_j y_j)+ f_i(\boldsymbol{v}_1) \right)\] 
where $\boldsymbol{v}_1=(x_1,\ldots,x_m)$ and $\boldsymbol{v}_2=(y_1,\ldots,y_m)$ is at most $\sum_{i=0}^{l} \binom{m}{i}$. To see this, first we notice that since $\boldsymbol{v}_2$ is a $\{0,1\}$-vector, we have $y_j^n=y_j$ for every $j$ and $n\ge 1$. Hence if we let $y_S=\prod_{t\in S}y_t$ for $S\subseteq [m]$, then expanding the right-hand side of the above equation, we may write
\[T(\boldsymbol{v}_1,\boldsymbol{v}_2)=\sum_{S\subseteq [m],|S|\le l}g_S(\boldsymbol{v}_1)y_S.\] 
Now, as it can be seen from the last equality, each such subset $S$ 
contributes one rank-1 term, and hence the result follows from the fact that 
the number of subsets $S \subseteq [m]$ with $|S| \leq l$ is exactly:
\[
\sum_{i=0}^l \binom{m}{i}.
\]
\end{example}

The following lemma is at the heart of the slice-rank method. 
\begin{lemma}[T. Tao \cite{tao_notes_2016}]
    \label{slicerank}
	Let $A$ be a finite set and $F$ be a field. If a $k$-tensor $T\colon A^k \rightarrow F$ is diagonal, then $sr(T)$ is equal to $|A|$.
\end{lemma}

\subsection{Triangular $2$-Tensor}
Notice that when $T$ is a tensor of order two, we can think of $T$ as a square matrix of order $|A|$ and it is not that hard to see that the notion of the slice rank coincides with the usual notion of the rank of a matrix. Considering this connection, one may wonder whether it is possible to extend some known facts about the rank of matrices to the slice-rank of tensors. It turns out that indeed this is the case if we consider triangular matrices where we have the fact that the rank of a triangular matrix is at least equal to the number of its non-zero diagonal entries. In order to extend this fact to the tensors, we first need a proper notion of a triangular tensor. This can be done naturally for 2-tensors if we equip $A$ with a total order. 

\begin{definition}\label{diagonal-lower triangular 2-tensor}
Let $F$ be a field, $(A,\le)$ be a totally ordered set, and let $T\colon A^2\longrightarrow F$ be a $2$-tensor. The $2$-tensor $T$ is called lower-triangular if  $T(u,v)=0$ for $u\neq v$ and  $u\le v$. Similarly, an upper-triangular $2$-tensor $T$ is a tensor with the property that $T(u,v)=0$ for $u\neq v$ and $u\ge v$. 
\end{definition}

This definition and the fact that the rank of a triangular matrix with non-zero diagonal entries is equal to its slice-rank leads to the following generalization of \cref{slicerank}.

\begin{lemma}\label{slice-rank-triangular}
Let $(A,\le)$ be a totally ordered finite set and $F$ be a field. If a $2$-tensor $T\colon A^2 \rightarrow F$ is triangular with non-zero diagonal entries, then $sr(T)$ is equal to $|A|$.
\end{lemma}

The above lemma will be our main tool in establishing our main  set-intersection results. 


\subsection{Triangular $k$-Tensor}

In the previous section, we defined the notion of a triangular 2-tensor. One possible way to extend that definition to higher-order tensors is as follows.
\begin{definition}
 Let $(A,\le)$ be a totally ordered finite set, and $F$ be a field. A $k$-tensor $T:A^k\to F$ is $i$-triangular if $T(x_1,\ldots,x_k)=0$ when $(x_1,\ldots,x_k)$ is off-diagonal and $x_i\le x_j$ for all $j\ne i$.    
\end{definition}

\begin{remark}
In \cite{AmanovYeliussizov2021}, Amanov and Yeliussizov have defined the notion of higher-order tensors in $P$-echelon form and have proved some results similar to the slice-rank lemma. Taking a $k$-tensor $T$ on a totally ordered finite set $A$, they equip $[k]$ with a partial order $P$ with connected Hasse diagram and call $T$ to be in $P$-echelon form if $T(x_1,x_2,\ldots,x_k)=0$ unless $x_i\le x_j$ for $i<_Pj$. Their definition and results can be considered as a generalization of the triangular slice-rank lemma for 2-tensors to higher-order tensors though it has not been explicitly stated in their paper. In our definition, we equip $[k]$ with the partial order $Q_i=\{(i,j)\,\vert\, i\le j \;\text{for all}\; j\}$ and define an $i$-triangular tensor. Our definition is different from their definition since they define a $k$-tensor $T$ to be in $Q_i$-echelon form if $T(x_1,x_2,\ldots,x_k)=0$ unless $x_{\alpha}\le x_{\beta}$ for ${\alpha}<_{Q_i}{\beta}$.
\end{remark} 


The following lemma is the slice-rank lemma of $i$-triangular tensors.

\begin{lemma}
Let $(A,\le)$ be a totally ordered finite set, and $F$ be a field. Furthermore, let $T\colon A^k \rightarrow F$ be an $i$-triangular $k$-tensor with non-zero diagonal entries for some $i \in [k]$.  Then:
\begin{itemize}
    \item[(i)] $sr(T)=|A|$ when $k$ is an even number, and
    \item[(ii)] $sr(T) \ge |A|/(p-1)$ when $k$ is an odd number and $F$ is a field of positive characteristic $p$. 
    \end{itemize}
\end{lemma}
\begin{proof}
Let $A=[n]$, and let $1 < 2 < \cdots < n$ be the total order of $[n]$. Without loss of generality we may assume that $T$ is $1$-triangular, i.e.,\ $T(x_1,\ldots,x_k)=0$ if $(x_1,\ldots,x_k)$ is off-diagonal and $x_1 \le x_i$ for all $i>1$.  To prove the claim we need the notion of the hyper-determinant of a tensor. Cayley's first hyper-determinant  for the $k$-tensor $T\colon [n]^k \rightarrow F$ (see~\cite{AmanovYeliussizov2021} for more details on hyper-determinants) is
\[
    \det(T) := \sum_{\sigma_2, \ldots, \sigma_k \in S_n} \text{sgn}(\sigma_2 \cdots \sigma_k) \prod_{i=1}^{n} T(i, \sigma_2(i), \ldots, \sigma_k(i)) 
\]
where as usual $\text{sgn}(\cdot)$ denotes the signature or parity of a permutation. We claim that 
\begin{equation}\label{Diagonal-hyper-det}
\operatorname{det}(T)=\prod_{x\in A=[n]} T(x,\ldots,x).
\end{equation}
In order to prove our claim, it suffices to show that for every collection of $k$ permutations $\sigma_2,\ldots,\sigma_k \in S_n$, if one of them is different from the identity permutation, then 
\begin{equation}\label{0-Tensor}
    \prod_{i=1}^n T(i,\sigma_2(i),\ldots,\sigma_k(i))=0. 
\end{equation}
Suppose at least one of the permutations $\sigma_2,\ldots,\sigma_k \in S_n$ is different from the identity. Then there are some elements not fixed by all the permutations. If we let $m$ be the smallest such an element, then $\sigma_i(m)\ge m$ for every permutation $\sigma_i$. But since $T$ is $1$-triangular, this would imply that 
\[
T(m,\sigma_2(m),\ldots,\sigma_k(m))=0
\]
which in turn implies ~\eqref{0-Tensor}. Since by our assumptions $T$ has non-zero diagonal entries, from~\eqref{Diagonal-hyper-det} we deduce that $\det(T)$ is non-zero. But using Theorems~4.3 and~4.5 of \cite{AmanovYeliussizov2021}, this would imply $\operatorname{sr}(T)=|A|$ and $\operatorname{sr}(T)\ge |A|/(p-1)$ for even $k$ and $F$ being of positive characteristic $p$, respectively. This finishes the proof. 
\end{proof}

\subsection{Preliminary Results}

In this section, we gather some results which will be used later.

\begin{lemma}[Complement Lemma]
    Let $p$ be a prime number dividing $n$, $\alpha = (0, \ldots, 0, L)$ be a $k$-intersection profile where $L = -L$ for $L \subseteq \mathbb{Z}_p$  , and let $\mathcal{F} = \{A_1, \ldots, A_m\}$ be an $[n,p,\alpha]$-family. Furthermore, for any $i$, $1 \leq i \leq m$, let  
    \[\mathcal{F}^i: = (\mathcal{F}-\{A_i\})\cup \{A_i^c\}\]
    where $A_i^c$ denotes the complement of the set $A_i$ in the set $[n]$. Then $\mathcal{F}^i$ is an $[n,p,\alpha]$-family.  Notice that $\mathcal{F}^i$  is obtained by replacing $A_i$ with $A_i^c$.
\end{lemma}

\begin{proof}
     We first observe that if $A$ and $B$ are any two subsets of $[n]$, then from $|A\cap B|+|A^c\cap B|=|B|$, it follows that if $|B|\equiv 0\pmod p$, then $|A\cap B|\equiv-|A^c\cap B|\pmod p$, and in particular $|A^c\cap B|\equiv 0\pmod p$ if $|B|\equiv|A\cap B|\equiv 0\pmod p$.  

     In order to prove that $\mathcal{F}^i$ is an $[n,p,\alpha]$-family, it suffices to show that:
     \begin{itemize}
         \item[(i)] $|A_i^c|\equiv 0\pmod p$,
         \item[(ii)] $|A_i^c\cap A_{r_1} \cap \cdots \cap A_{r_\ell}|\equiv 0\pmod p$ for any $1 \leq r_1 < \cdots < r_\ell \leq m$ with $l<k-1$ and $r_j\neq i$ for any $j$, $1\leq j\leq l$, and 
         \item[(iii)] $|A_i^c\cap A_{r_1} \cap \cdots \cap A_{r_{k-1}}|  \,\mathrm{mod}\, p \in L$ for any $1 \leq r_1 < \cdots < r_{k-1} \leq m$.     
         \end{itemize}
    But it is straightforward to see that (i), (ii) and (iii) follow from the observation made above by letting $B=[n]$, $B=A_{r_1} \cap \cdots \cap A_{r_\ell}$, and $B=A_{r_1} \cap \cdots \cap A_{r_{k-1}}$, respectively, and using the facts that $n\equiv |A_i|\equiv 0\pmod p$ and $L=-L$. 
\end{proof}
The following is a refinement of the Trace Lemma of O'Neill and Johnston~\cite[Lemma~1]{ONeill2025}.
\begin{lemma}[Trace Lemma]
Let $\alpha = (L_1, \ldots, L_k)$ be a $k$-intersection profile  where $L_t \cap L_{t+1} = \emptyset$  for some $t > 1$, and let $\alpha^\prime=(L_2,\ldots,L_k)$.
If 
\(
\mathcal{F} = \{A_1,\ldots,A_m\} 
\)
is an $[n,p,\alpha]$-family and $n > t$, then 
\[
\mathcal{F}_\gamma = \{ A_i\cap A_{\gamma}|i\neq \gamma\}
\]
is an $[|A_{\gamma}|,p,\alpha^\prime]$-family of size $m-1$.
\end{lemma}

\begin{proof}
The fact that $\mathcal{F}_\gamma$ is an $[|A_{\gamma}|,p,\alpha^\prime]$-family is easy to check. We need to show that $|\mathcal{F}_{\gamma}|=m-1$, i.e., $A_i\cap A_\gamma\neq A_j\cap A_\gamma$ if $i\neq j$. Suppose it is not the case and $A_i\cap A_{\gamma}=A_j\cap A_{\gamma}$ for some $i\neq j$ which implies $ A_i \cap A_j\cap A_\gamma=A_i\cap A_\gamma$. Then, using the fact that $m\ge k\ge t+1$ and taking pairwise distinct  $A_{r_1}, A_{r_2}, \ldots, A_{r_{t-2}}$ from $\mathcal{F}$ which are different from $A_i,A_j$ and $A_\gamma$, it follows that
\[A_i \cap A_j\cap A_\gamma \cap (\bigcap_{\ell=1}^{t-2} A_{r_\ell}) = A_i\cap A_\gamma \cap  (\bigcap_{\ell=1}^{t-2} A_{r_\ell}).\]
Denoting the left-hand side of the above equation by $U$ and its right-hand side by $V$, we have $|U| \equiv |V| \pmod{p}$. But we have $|U| \pmod{p} \in L_{t+1}$ and $|V|\pmod{p} \in L_t$ which contradicts the fact that $L_t \cap L_{t+1} = \emptyset$. This completes the proof.
\end{proof}

\section{Main Results}
\label{Main}
\subsection{New proofs using the triangular tensor}~\label{New-proofs}
We begin this section with the following lemma, which can be proved easily using the definitions of characteristic vectors and lexicographic order. 

\begin{lemma}\label{order-lemma}
 Let $\mathcal{F}\subseteq 2^{[n]}$, and for $X,Y\in\mathcal{F}$ let $X\le Y$ if and only if $\boldsymbol{v}_{X}\le \boldsymbol{v}_{Y}$, lexicographically from the left-hand side. Then: 
 \begin{itemize}
     \item[(i)] $\le$ is a total order on $\mathcal{F}$,
     \item[(ii)] if $X\subseteq Y$, then $X\le Y$,
     \item[(iii)] if $X\le Y$, $\boldsymbol{v}_X=(x_1,\ldots,x_n)$ and $\boldsymbol{v}_Y=(y_1,\ldots,y_n)$, then $x_1y_1=x_1$.
 \end{itemize}
 \end{lemma}

 The following is a new proof of Snevily's theorem with modular restrictions, which is a generalization of the odd-town theorem.

\begin{theorem}[Snevily \cite{Snevily1}]
\label{snevily}
Let \( p \) be a prime number, and  let $\alpha=(L,K)$ where \( L\) and \( K\) are two disjoint subsets of  \(\mathbb{Z}_p\) with $|L| = s$.
Furthermore, let \( \mathcal{F} = \{A_1,\ldots,A_m\} \) be an $[n,p,\alpha]$-family. Then \( |\mathcal{F}| \le \sum_{i=0}^{s} \binom{n-1}{i}. \)
\end{theorem}
\begin{proof}
Let $L=\{l_1,\ldots,l_s\}$, and  let $T(X,Y)$ be the following 2-tensor from $\mathcal{F} \times \mathcal{F}$ to $\mathbb{Z}_p$ where $\boldsymbol{v}_X=(x_j)_{j=1}^n$ and $\boldsymbol{v}_Y=(y_j)_{j=1}^n$ are the characteristic vectors of the sets $X, Y \in \mathcal{F}$, respectively:
\[
T(X, Y) = \prod_{i=1}^s \left(( \sum_{j=1}^n x_j y_j) - l_i\right) = \prod_{i=1}^s \left(x_1y_1 + (\sum_{j=2}^n x_jy_j) - l_i\right).
\]

Using the facts that $|X|=\sum_{j=1}^nx_j^2=\sum_{j=1}^nx_j$,  $|X\cap Y|=\sum_{j=1}^nx_jy_j$ and the fact that $L$ and $K$ are disjoint, it follows that $T$ is a diagonal tensor with non-zero diagonal entries. This fact leads us to define a new tensor $T'$ using the right-hand side of the above equation:
\[
T'(X,Y) = \prod_{i=1}^s \left(x_1 + (\sum_{j=2}^n x_jy_j) - l_i \right).
\]

Assuming that $\mathcal{F}$ is totally ordered by $\le$ as in the above lemma, by the same lemma $x_1y_1=x_1$ whenever $X\le Y$, and hence $T'(X,Y)=T(X,Y)$ whenever $X\le Y$. This implies that $T'$ is lower-triangular with non-zero diagonal entries. Rearranging the terms of $T'(X,Y)$, we may write
\[
T'(X,Y) =  \prod_{i=1}^s \left( (\sum_{j=2}^n x_j y_j) + f_i(X) \right) 
\]
where $f_i(X)$ is a function depending only on $X$. Now, from \cref{innerproduct-slice rank}, we have    
\[
\operatorname{sr}(T')\leq \sum_{i=0}^{s} \binom{n - 1}{i},
\]
and hence applying \cref{slice-rank-triangular} finishes the proof. 
\end{proof}
\begin{remark}
    Snevily~\cite{snevily2003sharp} has conjectured that the stronger bound $|\mathcal{F}|\le \binom{n}{s}$ is true in the above theorem. 
\end{remark}
In the sequel, we give a new proof of the following theorem, which appears as Lemma~2.7 in~\cite{liu2024snevily}.
\begin{theorem}[Liu \emph{et al.}~\cite{liu2024snevily}]
    Let $\mathcal{A} = \{ A_1,\ldots,A_m\}$, $\mathcal{B} = \{B_1,\ldots,B_m\}$ be two families of subsets of $[n]$. Furthermore,
    Let $\mathcal{L} = \{l_1,\ldots,l_k\}$ be a set of non-negative integers. Suppose $|A_r \cap B_s| \in \mathcal{L}$ for $r\neq s$, $A_r \subseteq B_r$, and $|A_r| \notin \mathcal{L}$ for $1\leq r \leq m$, then $m \leq \sum_{i=0}^k {n-1 \choose i}$.
\end{theorem}
\begin{proof}
The proof is very similar to the proof of the previous theorem. Let $T$ be a $2$-tensor form $\mathcal{A}\times \mathcal{A}$ to $\mathbb{R}$ such that for every $(r,s)\in [m]\times [m]$,
    \[
T(A_r, A_s) = \prod_{i=1}^k \left(( \sum_{j=1}^n x_j y_j) - l_i\right) = \prod_{i=1}^k\left(x_1y_1 + (\sum_{j=2}^n x_jy_j) - l_i\right)
\]
 where $(x_j)_{j=1}^n$ and $(y_j)_{j=1}^n$ are the characteristic vectors of $A_r\in \mathcal{A}$ and $B_s\in \mathcal{B}$, respectively.  Using the facts that $|A_r|=\sum_{i=1}^nx_i^2=\sum_{i=1}^nx_i$,  $|A_r\cap B_s|=\sum_{i=1}^nx_iy_i$ and the facts that $|A_r \cap B_s| \in \mathcal{L}$ for $r\neq s$,  and $|A_r| \notin \mathcal{L}$ for $1\leq r \leq m$, it follows that $T$ is a diagonal tensor with non-zero diagonal entries. 

Now, the same as the proof of the previous theorem, we define a new tensor $T'$ where
\[
T'(A_r,A_s) = \prod_{i=1}^k \left(x_1 + (\sum_{j=2}^n x_jy_j) - l_i \right).
\] Assuming that $\mathcal{A}$ is totally ordered by $\le$ as in \cref{order-lemma}, by the same lemma $x_1y_1=x_1$ whenever $A_r\le A_s\le B_s$ (notice that $A_s\subseteq B_s$), and hence $T'(A_r,A_s)=T(A_r,A_s)$ whenever $A_r\le A_s$.  The rest of the proof is completely similar to the proof of the theorem of Snevily.

\end{proof}

Here is a new proof of the Frankl-Wilson theorem.

\begin{theorem}[Frankl-Wilson \cite{FranklWilson1981}]
 Let \( L=\{l_1,\ldots,l_s\}\)  be a set of $s$ non-negative integers, and let \( \mathcal{F} = \{A_1,\ldots,A_m\} \) be a collection of subsets of $[n]$ such that $|A_i \cap A_j| \in L$ for $i \neq j$.
Then \( |\mathcal{F}| \le \sum_{i=0}^{s} \binom{n}{i}. \)   
\end{theorem}
\begin{proof}
Let $\le$ be a binary relation on $\mathcal{F}$ satisfying
\begin{itemize}
    \item[(i)] $X\le X$, and
    \item[(ii)] if $X\neq Y$, then $X\le Y$ if either $|X|<|Y|$ or $|X|=|Y|$ and $v_X<v_Y$ lexicographically.  
\end{itemize}
Then, trivially, $\le$ is a total order on $\mathcal{F}$. 

Now, let \( T(X, Y) \) be the following 2-tensor from \( \mathcal{F} \times \mathcal{F} \) to \( \mathbb{R} \)
    \[
        T(X, Y) = \prod_{i=1}^s \left( (\sum_{j=1}^n x_j y_j) - l_i + \delta_{l_i, |Y|} \right)
    \]  
where \( \boldsymbol{v}_X = (x_j)_{j=1}^n \) and \( \boldsymbol{v}_Y = (y_j)_{j=1}^n \) are the characteristic vectors of the sets \( X, Y \in \mathcal{F} \), respectively, and \( \delta_{l_i,|Y|} \) is the Kronecker delta function, i.e., $\delta_{l_i,|Y|}=1$ if and only if $|Y|=l_i$. We claim that \( T \) is lower-triangular with non-zero diagonal entries with respect to the total order $\le$. 
To prove the claim, firstly, we notice that for every $X\in \mathcal{F}$, whenever $|X|\neq l_i$, then $\sum_{j=1}^n x_j x_j - l_i + \delta_{l_i, |X|}=|X|-l_i\neq 0$, and whenever $|X|=l_i$, then $\sum_{j=1}^n x_j x_j - l_i + \delta_{l_i, |X|}=1\neq 0$, and hence $T(X,X)\neq 0$ for every $X\in\mathcal{F}$ which means that all diagonal entries are non-zero. Secondly, let $X< Y$. Then by the fact that $|X\cap Y|\in L$, we have $\sum_{j=1}^n x_j y_j - l_i=0$ for some $i$. But for the same $i$, $\delta_{l_i, |Y|}=0$ as when $X<Y$, then either $|X|<|Y|$ or $|X|=|Y|$ and in both cases since $|X\cap Y|=l_i$, we have $|Y|>l_i$. Thus we conclude that when $X<Y$, then for some $i$, $\sum_{j=1}^n x_j y_j - l_i + \delta_{l_i, |Y|}=0$, 
and hence $T(X,Y)=0$. Now, the result follows from Theorem~\ref{diagonal-lower triangular 2-tensor} and Example~\ref{innerproduct-slice rank}.
\end{proof}
\begin{remark}
 Snevily~\cite{snevily2003sharp} has shown that when the set $L$ in the above theorem consists of only positive integers, then \( |\mathcal{F}| \le \sum_{i=0}^s\binom{n-1}{i}. \) 
\end{remark}

Heged{\"u}s in~\cite{hegedus2024upper} has shown that the bound given in the theorem of Frankl-Wilson can be improved if the set $\mathcal{F}$ is equipped with a total order satisfying some extra conditions. Here we provide an alternative proof which is very similar to our proof of Frankl-Wilson.

\begin{theorem}[Heged{\"u}s]
With the assumptions as the theorem of Frankl-Wilson, suppose there exists $1\le r\le m$ such that
\begin{itemize}
    \item[(i)]$n\in A_i$ for every $1\le i\le r$,
    \item[(ii)]$n\notin A_i$ for $i>r$, and
    \item[(iii)]$|A_i|<|A_j|$ for every $1\le i<j\le m$. 
\end{itemize}
Then \( |\mathcal{F}| \le \sum_{i=0}^s\binom{n-1}{i}. \) 
\begin{proof}
    With the assumptions as above, we have \(x_{n}y_{n} = y_{n}\) whenever $X \leq Y$ where \( \boldsymbol{v}_X = (x_j)_{j=1}^n \) and \( \boldsymbol{v}_Y = (y_j)_{j=1}^n \) are the characteristic vectors of the sets \( X, Y \in \mathcal{F} \), respectively. Now, the tensor 
    \[
    T(X,Y) = \prod_{i=1}^{|L|} \left( (\sum_{j=1}^{n-1} x_j y_j) + y_n - l_i + \delta_{l_i, |Y|} \right)
    \]
    is triangular with which using the same arguments one can get the claimed bound.
    \end{proof}
\end{theorem}
\subsection{Generalization of the reverse odd-town theorem}

When $n$ is even and $p=2$, the bound given for the reverse odd-town problem by the theorem of Snevily is $n$ while the correct answer is $n-1$. We improve the theorem of Snevily in this special case to get a tight bound for the the reverse odd-town problem.  

\begin{theorem}
    \label{sharper}
    With the assumptions of \cref{snevily}, and additionally assuming \(K = \{0\}\) and \(L = -L\), the following holds:  
    if \(p \nmid n\), then  
    \(
        |\mathcal{F}| \le \sum_{i=0}^{|L|} \binom{n-1}{i},
    \)
    whereas if \(p \mid n\), the bound improves to  
    \(
        |\mathcal{F}| \le \sum_{i=0}^{|L|} \binom{n-2}{i}.
    \)
\end{theorem}
\begin{proof}
When \(p \nmid n\), the bound in \cref{snevily} applies directly, yielding  
\(|\mathcal{F}| \le \sum_{i=0}^{|L|} \binom{n-1}{i}\). Thus, it remains to show that if \(p \mid n\), then the bound can be improved to the claimed one.

Let $\alpha=(0,L)$, and let \(\mathcal{F}^\prime\) be a family obtained from \(\mathcal{F}\) by replacing each element of \(\mathcal{F}\) containing $n$ with its complement in $[n]$. Applying the complement lemma, \(\mathcal{F}^\prime\) is an \([n-1,p,\alpha]\)-family. 
Applying the above theorem we have  
\[
|\mathcal{F}| = |\mathcal{F}'| \le \sum_{i=0}^{|L|} \binom{n-2}{i}.
\]
\end{proof} 

The following theorem is a generalization of reverse odd-town problem where we are concerned about the higher order intersections and not just about the pairwise intersections. The theorem improves the bounds of $g_{(0,0,\star)}(n)$ given in~\cite[Section~5.6]{ONeill2025}.
\begin{theorem}
\label{Generalized-reverse-odd-town}
    Let $p$ be a prime number and $k>2$ be an integer. Furthermore let  $\alpha = (0, \ldots, 0, L)$ be a $k$-intersection profile where $L \subseteq \mathbb{Z}_p$, $0\notin L$, and $L = -L$. Then for any $[n,p,\alpha]$-family $\mathcal{F}$, the following holds: 
    \begin{itemize}
    \item[(i)] if $p | n$, then
    \[
        |\mathcal{F}| \leq \sum_{i=0}^{|L|} {\left\lfloor\dfrac{n}{2^{k-2}}\right\rfloor-2 \choose i} + k-2,
    \]    
    
    \item[(ii)]  if $p \nmid n$, then we have the weaker bound
    \[
        |\mathcal{F}| \leq \sum_{i=0}^{|L|} {\left\lfloor\dfrac{n}{2^{k-3}}\right\rfloor-2 \choose i} + k-2.
    \]
    \end{itemize}
\end{theorem}
\begin{proof}
     When $p \mid n$, then by the Complement Lemma, for any $A \in \mathcal{F}$ with $|A| > \frac{n}{2}$,  the family $\mathcal{F}'$ obtained by  replacing $A$ with $A^c$ in $\mathcal{F}$ is an $[n,p,\alpha]$-family.
    Hence, without loss of generality, we may assume that all the sets in $\mathcal{F}$ are of size at most $n/2$. Now, let $A$ be any set in $\mathcal{F}$, and let $\mathcal{F}' = \{A \cap B : B \in \mathcal{F} \setminus \{A\}\}$. Then, by the Trace Lemma, the family $\mathcal{F}'$ is an  $[|A|,p,\alpha']$-family, where $\alpha' = (0,\ldots,0,L)$ is a $(k-1)$-intersection profile, and $|A| \equiv 0 \pmod{p}$.
    Moreover, $|\mathcal{F}'| = |\mathcal{F}| - 1$ and $|A| \leq \lfloor n/2 \rfloor$. Repeating this argument recursively, the problem gets reduced to bounding the size of the $[n',p,(0,L)]$-family $\mathcal{G}$, where $n' \le \left\lfloor \dfrac{n}{2^{k-2}} \right\rfloor$. But, then applying \cref{sharper} we have
    \[
        |\mathcal{G}| \leq \sum_{i=0}^{|L|} \binom{\left\lfloor \dfrac{n}{2^{k-2}} \right\rfloor - 2}{i},
    \] 
    and hence
    \[
        |\mathcal{F}| \leq \sum_{i=0}^{|L|} \binom{\left\lfloor \dfrac{n}{2^{k-2}} \right\rfloor - 2}{i} + k - 2.
    \] 
    Now, we prove the case where $p\nmid n$. Let $\mathcal{F}=\{A_1,\ldots,A_m\}$ be an $[n,p,\alpha]$-family.  Then, using the Trace Lemma
    \[
        \mathcal{F}_1=\{A_2\cap A_1,\ldots,A_m\cap A_1\}
    \]
    is an $[|A_1|,p,\alpha']$-family of size $m-1$ where $\alpha'=(0,0,\ldots,L)$ is a $(k-1)$-intersection profile . But, then $p\mid|A_1|$, and hence the result follows from the previous case.
\end{proof}

\bibliographystyle{plain}
\bibliography{refs}
\end{document}